\documentclass[journal,twoside,web]{ieeecolor}
\usepackage{generic}
\usepackage{cite}
\usepackage{amsmath,amssymb,amsfonts}
\usepackage{algorithmic}
\usepackage{graphicx}
\usepackage{textcomp}
\usepackage{algorithmic}
\usepackage{tikz}
\graphicspath{{images/}}
\usepackage{graphicx}
\usepackage{epstopdf}
\def\BibTeX{{\rm B\kern-.05em{\sc i\kern-.025em b}\kern-.08em
    T\kern-.1667em\lower.7ex\hbox{E}\kern-.125emX}}
\def\J{{\bf 1}}

\DeclareMathOperator{\Col}{Col}
\DeclareMathOperator{\Row}{Row}

\DeclareMathOperator{\lcm}{lcm}

\def\cal{\mathcal}

\def\diag{diag}
\def\ra{\rightarrow}

\def\lra{\leftrightarrow}

\def\d{\delta}

\def\D{\Delta}

\def\0{{\bf 0}}

\newtheorem{thm}{Theorem}[section]
\newtheorem{dfn}[thm]{Definition}
\newtheorem{prp}[thm]{Proposition}
\newtheorem{exa}[thm]{Example}

\newtheorem{rem}[thm]{Remark}
\newtheorem{alg}[thm]{Algorithm}

\begin{document}
\title{State Feedback Stabilization of Generic Logic Systems via Ledley Antecedence Solution}
\author{Yingzhe Jia, \IEEEmembership{Student Member, IEEE}, Daizhan Cheng, \IEEEmembership{Fellow, IEEE} and Jun-e Feng
\thanks{This work is supported partly by the National Natural Science Foundation of China (NSFC) under Grants 61773371, 61733018 and 61877036.}
\thanks{Yingzhe Jia and Jun-e Feng are with School of Mathematics, Shandong University, Jinan 250100, P. R. China (e-mail for Feng: fengjune@sdu.edu.cn, e-mail for Jia: yingzhe.jia@postgrad.manchester.ac.uk). }
\thanks{Daizhan Cheng is with the Key Laboratory of Systems and Control, Academy of Mathematics and Systems Sciences, Chinese Academy of Sciences, Beijing 100190, P. R. China (e-mail dcheng@iss.ac.cn).}
\thanks{Corresponding author: Jun-e Feng. Tel.: +86 531 88364652.}}

\maketitle

\begin{abstract}
In this paper, the application of Ledley antecedence solutions in designing state feedback stabilizers of generic logic systems has been proposed. To make the method feasible, two modifications are made to the original Ledley antecedence solution theory: (i) the preassigned logical functions have been extended from being a set of equations to an admissible set; (ii) the domain of arguments has been extended from the whole state space to a restricted subset. In the proposed method, state feedback controls are considered as a set of extended Ledley antecedence solutions for a designed iterative admissible sets over their corresponding restricted subsets. Based on this, an algorithm has been proposed to verify the solvability, and simultaneously to provide all possible state feedback stabilizers when the problem is solvable. All stabilizers are optimal, which stabilize the logic systems from any initial state to the destination state/state set in the shortest time. The method is firstly demonstrated on Boolean control networks to achieve point stabilization. Then, with some minor modifications, the proposed method is also proven to be applicable to set stabilization problems. Finally, it is shown that in $k$-valued and mix-valued logical systems, the proposed method remains effective.
\end{abstract}


\section{Introduction}
\label{sec:introduction}
\IEEEPARstart{T}{he} problem of stability and stabilization is one of the most important issues  in studying dynamic (control) systems. It is also a fundamental topic for Boolean (control) networks as well as for $k$-valued or  mix-valued (control) networks.  For statement ease, in the following text  a logical network could be a Boolean, $k$-valued, or  mix-valued network. The systematic investigation of stability of logical networks may be traced back to F. Robert \cite{rob86}, where it was called the convergence of discrete iteration. The vector metric was introduced and used to provide some sufficient conditions for convergence of discrete iterations.

Boolean network (BN) was introduced firstly by Kauffman \cite{kau69,kau93} to formulate cellular networks. To manipulate BN, the Boolean control network (BCN) was merged naturally \cite{hua00}. Since then, the stability and stabilization become an interesting and challenging topic to study. Later on, Cheng and his colleagues proposed a new matrix product called semi-tensor product (STP) of matrices, and used STP to convert a logical network into an algebraic form called the algebraic state space representation (ASSR) of logical networks \cite{che11,che12}. Motivated by ASSR, the investigation of Boolean (control) networks has been developed quickly. Many fundamental control problems about logical networks have been studied, for instance, controllability and observability\cite{che09,for13b}, state space decomposition and disturbance decoupling \cite{che11c,fu18}, optimization and optimal control \cite{zha11,for14}, with some applications to finite automata \cite{xu12,yan15}, coding \cite{zho19}, just to mention a few.  A general outline on recent development of logical networks via STP/ASSR and its applications can be discovered by some survey papers \cite{for16,lu17,muh16,li18,zhao20}.

The stability of logical networks have also been studied via ASSR widely. Three basic approaches have been developed: (i) incidence-matrix-based stability analysis \cite{che11b}; (ii) transition-matrix-based stability analysis \cite{che11,for13}; (iii) Lyapunov function-based analysis \cite{li17}. Of course, stability is closely related to the stabilization of logical control networks, because stabilization is to find a suitable controller, which is commonly called the stabilizer, to make the controlled system being stable.

Various stabilization problems have also been investigated via STP and ASSR. For instance,  set stability and stabilization were firstly proposed and investigated by \cite{guo15};  stability and stabilization of BNC with delay have been discussed by several authors, say \cite{men19};
state feedback stabilization with design technique was proposed by \cite{li17b}; partial stability and stabilization were investigated by \cite{che16}; robust stability and stabilization of BCN with disturbances were considered in \cite{zho17}.

This paper investigates the state feedback stabilization of generic logic systems. The technique proposed in this paper is based on Ledley antecedence solution.
 The antecedence/consequence solution for a set of logical equations was firstly proposed by R.S. Ledley \cite{led55,led73}. It has been applied to several logical problems such as error diagnosis of disease, digital computational method in symbolic logic, digital circuit design, etc. We also refer to \cite{kim82} for a systematic description.  It was firstly applied to design of BCN by  \cite{qia18}. Using STP, some useful formulas were developed by \cite{qia18} to calculate anticedence/consequence solutions. The method of Ledley antecedence/consequence solution has also been used to design control of singular Boolean networks \cite{qipr}.

It is interesting that Ledley antecedence solution may provide a suitable state feedback stabilizer in a very natural and convenient way. To this end, two generalizations of Ledley antecedence solution have been done: (i) Ledley antecedence solution has been originally used with respect to a set of equalities, in which the preassigned functions equal to a set of constants respectively. In this paper, it has been developed such that the constants can be converted to an admissible set. (ii) Originally,  the domain of arguments for Ledley antecedence solution is over the whole state space. In this paper it has been developed into a restricted set.

The main advantages of this new approach are: (i) The algorithm is necessary and sufficient. It can verify whether the problem is solvable, and simultaneously it can also produce a state feedback stabilizer provided the problem is solvable. (ii) The stabilizer provides an optimal solution, which means each point can reach its destination point/set in the shortest time. (iii) All the feasible state feedback stabilizers can be found with one search. (iv) With a mild revision, the method is applicable to set stabilization. In addition, the computations involved are easy and straightforward.

Before ending this section, we give a brief list for notations:

\begin{enumerate}


\item  ${\cal M}_{m\times n}$: the set of $m\times n$ real matrices.

\item $\Col(M)$ ($\Row(M)$): the set of columns (rows) of $M$. $\Col_i(M)$ ($\Row_i(M)$): the $i$-th column (row) of $M$.

\item ${\cal D}:=\{0,1\}$.

\item ${\cal D}_k:=\left\{0,\frac{1}{k-1},\cdots,\frac{k-2}{k-1},1\right\}$, $k\geq 2$.

\item $\d_n^i$: the $i$-th column of the identity matrix $I_n$.

\item $\D_n:=\left\{\d_n^i\vert i=1,\cdots,n\right\}$; $\D :=\D_2$.

\item ${\bf 0}_{\ell}:=(\underbrace{0,0,\cdots,0}_{\ell})^{\mathrm{T}}$.

\item A matrix $L\in {\cal M}_{m\times n}$ is called a logical matrix
if $\Col(L)\subset \D_m$.
Denote by ${\cal L}_{m\times n}$ the set of $m\times n$ logical
matrices.

\item If $L\in {\cal L}_{n\times r}$, by definition it can be expressed as
$L=[\d_n^{i_1},\d_n^{i_2},\cdots,\d_n^{i_r}]$. For the sake of
compactness, it is briefly denoted as $
L=\d_n[i_1,i_2,\cdots,i_r]$.

\item $\d_n\{a_1,a_2,\cdots, a_n\}:=\{\d_n^{a_1},\d_n^{a_2},\cdots,\d_n^{a_n}\}$.

\item $\ltimes$: Semi-tensor product of matrices.
\item Let $A\in {\cal M}_{p\times n}$ and $B\in {\cal M}_{q\times n}$. Then $A*B\in {\cal M}_{pq\times n}$ is the Khatri-Rao product of matrix, satisfying \cite{che11}
$$
\Col_j(A*B)=\Col_j(A)\ltimes \Col_j(B),\quad j=1,\cdots,n.
$$
\item Denote by ${\cal B}_{m\times n}$ the set of $m\times n$ Boolean matrices.

\item $A,B\in {\cal M}_{m\times n}$, $A\geq B$ means that  $A_{i,j}\geq B_{i,j}$, $\forall i,j$.

\item $PR_k$: power reducing matrix, which is defined as
$$
PR_k=\diag(\d_k^1,\d_k^2,\cdots,\d_k^k)\in {\cal M}_{k^2\times k}.
$$

\end{enumerate}

The rest of this paper is organized as follows: Section II provides some necessary preliminaries, including two parts: (i) STP of matrices and the matrix express of logical functions; (ii) Ledley antecedence/consequnce solution to a set of logical functions. Section III considers the state feedback stabilization of BCN to a pre-assigned point. In Section III the Ledley antecedence/consequence solution has been generalized in two ways: (i) the set of equalities is generalized to a set of inclusions; (ii) the antecedence/consequence solution has been generalized to subset antecedence/consequence. Section IV uses the generalized antecedence to construct state feedback stabilizer. Section V discusses the stabilization to a pre-assigned subset. The algorithm for point stabilization has been extended to a similar algorithm for set stabilization. Section VI provides an example to show that the technique developed is also applicable to $k$-valued or mix-valued logical networks. In section VII, some concluding remarks of the paper are given.

\section{Preliminaries}

\subsection{Matrix Expression of Logical Functions}
We first recall STP, which is the fundamental tool for algebraic expression of logical functions.

\begin{dfn}\label{d2.1.1} \cite{che11} Let $A\in {\cal M}_{m\times n}$,  $B\in {\cal M}_{p\times q}$, and the least common multiple $\lcm(n,p)=t$.
Then the SPT of $A$ and $B$ is defined by
\begin{align}\label{2.1.1}
A\ltimes B:=\left(A\otimes I_{t/n}\right)\left(B\otimes I_{t/p}\right),
\end{align}
where $\otimes$ is Kronecker product.
\end{dfn}

Note that STP is a generalization of classical matrix product. That is, when $n=p$, STP degenerates to classical matrix product. Throughout this paper the default matrix product is STP, and in most cases the symbol $\ltimes$ is omitted.

$X$ is called a logical variable if $X\in {\cal D}$. The vector expression of $X$, denoted by $x=\vec{X}$, is
$$
x=\vec{X}:=\begin{bmatrix}X\\ 1-X\end{bmatrix}\in \D.
$$

Note that as a convention, we always use capital letters, such as $X_i$, $U_j$, for logical variables, and use their lower case, such as $x_i$, $u_j$, for their vector forms.

\begin{dfn}\label{d2.1.2} A mapping $\phi:{\cal D}\ra {\cal D}$ is called a unary operator.  A mapping $\phi:{\cal D}^2\ra {\cal D}$ is called a binary operator.  A mapping $\phi:{\cal D}^n\ra {\cal D}$ is called a (classical) logical function of $n$-variables. A classical logical function is also called a Boolean function.
\end{dfn}

So the unary operators are one-variable Boolean functions, and the binary operators are two-variable Boolean functions.
In vector form, for each operator there is a logical matrix such that the logical expression can be converted into an algebraic expression, called the ASSR expression.

Consider unary operator negation (~$\neg$)
\begin{align}\label{2.1.2}
\neg x=\begin{bmatrix}0&1\\1&0\end{bmatrix}x=\d_2[2,1]x:=M_nx,
\end{align}
where ~$M_n\in {\cal L}_{2\times 2}$ is called the structure matrix of negation.

Similarly, for a binary operator $\sigma$, we can also find its structure matrix $M_{\sigma}\in {\cal L}_{2\times 4}$, such that
$$
x\sigma y=M_{\sigma}xy.
$$
For instance,
\begin{itemize}
\item conjunction (~$\wedge$):
\begin{align}\label{2.1.3}
M_c=\d_2[1,2,2,2].
\end{align}
\item  disjunction (~$\vee$):
\begin{align}\label{2.1.4}
M_d=\d_2[1,1,1,2].
\end{align}
\item conditional (~$\ra$):
\begin{align}\label{2.1.5}
M_i=\d_2[1,2,1,1].
\end{align}
\item biconditional (~$\lra$):
\begin{align}\label{2.1.6}
M_e=\d_2[1,2,2,1].
\end{align}
\item exclusive or (~$\bar{\vee}$):
\begin{align}\label{2.1.7}
M_p=\d_2[2,1,1,2].
\end{align}
\end{itemize}

\begin{prp}\label{p2.1.3} \cite{che11} Let $F:{\cal D}^n\ra {\cal D}$ be a Boolean function. Then there exists a unique matrix, $M_{F}\in {\cal L}_{2\times 2^n}$ such that in vector form we have
\begin{align}\label{2.1.8}
f(x_1,\cdots,x_n)=M_F\ltimes_{i=1}^nx_i,
\end{align}
where $f$ is the vector form of $F$.
\end{prp}

\begin{prp}\label{p2.1.4} Let $x=\d_{k}^s$. Then
\begin{align}\label{2.1.9}
x^2=PR_kx,\quad x\in \D_k.
\end{align}
\end{prp}

\subsection{Ledley Antecedence/Consequence Solution}

The antecedence and consequence solutions of a given logical function were proposed and investigated firstly by Ledley \cite{led55,led73}. A systematic description with applications can be found in \cite{kim82}. The following definition comes from \cite{qia18}, which is a generalization of the original one given in \cite{kim82}.

\begin{dfn}\label{d2.2.1} Given a set of Boolean equations
\begin{align}\label{2.2.1}
F_i(X_1,\cdots,X_n,U_1,\cdots,U_m)=C_i,\quad i=1,\cdots,s,
\end{align}
 where $C_i\in {\cal D}$ are constants. A set of logical functions
\begin{align}\label{2.2.2}
U_j=G_j(X_1,X_2,\cdots,X_n), \quad j=1,2, \cdots,m
\end{align}
\begin{itemize}
\item[(i)] is called an antecedence solution of (\ref{2.2.1}), if (\ref{2.2.2}) implies (\ref{2.2.1});
\item[(ii)] is called a consequence solution of (\ref{2.2.1}), if (\ref{2.2.1}) implies (\ref{2.2.2}).
\end{itemize}
\end{dfn}

Note that in general for a given set of Boolean equations with $Z=\{Z_1,Z_2,\cdots,Z_q\}$ as their independent variables, an arbitrary partition
$$
Z=Z^1\bigcup Z^2
$$
is considered. The antecedence/consequence solution may have the form as
$$
Z_s=G_s(Z^1),\quad Z_s\in Z^2.
$$
For our purpose, (\ref{2.2.2}) is enough for control design.

Using vector form expression $x_i=\vec{X}_i$, $i=1,2,\cdots,n$, $u_j=\vec{U}_j$, $j=1,2,\cdots,m$, and denoting $x=\ltimes_{i=1}^nx_i$, $u=\ltimes_{j=1}^mu_j$, the algebraic forms of (\ref{2.2.1}) and (\ref{2.2.2}) are as the following (\ref{2.2.3}) and (\ref{2.2.4}) respectively:
\begin{align}\label{2.2.3}
M_Fux=c,
\end{align}
\begin{align}\label{2.2.4}
u=M_Gx.
\end{align}

Based on algebraic form (\ref{2.2.3}), \cite{qia18} proposed a matrix, called truth matrix, to describe the corresponding relation between $x$ and $u$. Roughly speaking, truth matrix shows when $x$ and $u$ are consistent. We use a simple example to describe this.

\begin{exa}\label{e2.2.2} Given
\begin{align}\label{2.2.5}
\begin{array}{l}
F_1(X_1,X_2,U)=(X_1\vee U)\ra X_2=\d_2^1,\\
F_2(X_1,X_2,U)=(X_1\wedge U)=\d_2^2,
\end{array}
\end{align}
it is easy to calculate that
\begin{align}\label{2.2.6}
M_F=\d_4[1,3,2,4,2,4,2,2],
\end{align}
and $c=\d_4^2$.
Next, we construct a matrix, where the columns represent different values of $x$ and rows represent different values of $u$. Then the table shows whether corresponding to each pair $(u,x)$  equation (\ref{2.2.3}) holds. If it is ``true", the corresponding entry is $1$, if it is ``false", the corresponding entry is $0$. Then we have Table \ref{tab.2.2.1}. We also say that the truth matrix is
\begin{align}\label{2.2.7}
T=\begin{bmatrix}
0&0&1&0\\
1&0&1&1
\end{bmatrix}.
\end{align}

\begin{table}[htbp] 

\centering \caption{Truth Matrix \label{tab.2.2.1}}
\doublerulesep 0.5pt
\begin{tabular}{c c c c c}
\hline\hline
$u\backslash x$&$\d_4^1$&$\d_4^2$&$\d_4^3$&$\d_4^4$\\
\hline
$\d_2^1$&0&0&1&0\\
$\d_2^2$&1&0&1&1\\
\hline\hline
\end{tabular}
\end{table}

\end{exa}

The following result is essential for verifying antecedence/consequence solutions.

\begin{thm}\label{t2.2.3} \cite{qia18}
Assume the truth table of equations (\ref{2.2.1}) is $T$, and the set of equations (\ref{2.2.2}) has its algebraic form as (\ref{2.2.4}).
Then
\begin{itemize}
\item[(i)] (\ref{2.2.2}) is an antecedence solution of (\ref{2.2.1}), if and only if,
\begin{align}\label{2.2.8}
M_G\leq T,
\end{align}
\item[(ii)] (\ref{2.2.2}) is a consequence solution of (\ref{2.2.1}), if and only if,
\begin{align}\label{2.2.9}
T\leq M_G.
\end{align}
\end{itemize}
\end{thm}

\section{Generalized Antecedence/Consequence Solutions}
By definition, the antecedence solution  generates a set of state feedback controls
$$
U_j=G_j(X_1,X_2,\cdots,X_n),\quad j=1,2,\cdots,m.
$$

To make this set of state feedback controls meeting our requirement, we need to generalize the previous concepts about antecedence solution to design stabilizer.

First, we replace a single state $C=(C_1,C_2,\cdots,C_s)$ by a set of states. Assume there exists a set $\Omega\subset {\cal D}_s$, called admissible set.
Replacing a constant set $C$, it is said that (\ref{2.2.1}) is true as long as $C=(C_1, C_2, \cdots, C_s)\in \Omega$. Then we can construct a truth matrix $T_{\Omega}$. In Theorem \ref{t2.2.3}, replacing $T$ by $T_{\Omega}$, it is easy to verify that  Theorem \ref{t2.2.3} remains available.

Second, we may consider a subset of states, $W\subset {\cal D}_n$, called restricted set. The concept of antecedence/consequence solutions can be generalized as follows:

\begin{dfn}\label{d3.1.1} Given a set of Boolean equations (\ref{2.2.1}), where $C=(C_1,C_2,\cdots,C_s)\in \Omega\subset {\cal D}_s$.  Consider a set of functions
\begin{align}\label{3.1.1}
U_j=G_j(X_1,X_2,\cdots,X_n), \quad j=1,2, \cdots,m.
\end{align}
\begin{itemize}
\item[(1)] (\ref{3.1.1}) is called a subset antecedence solution of (\ref{2.2.1}) with respect to $W$, if, as long as $(X_1,\cdots,X_n)\in W$, (\ref{3.1.1}) implies (\ref{2.2.1});
\item[(ii)] (\ref{3.1.1}) is called a subset consequence solution of (\ref{2.2.1}) with respect to $W$, if, as long as $(X_1,\cdots,X_n)\in W$,  (\ref{2.2.1}) implies (\ref{3.1.1}).
\end{itemize}
\end{dfn}

Note that at $T_{\Omega}$ each column corresponds to a distinct state $X$, the restriction of $T_{\Omega}$ on $W$, denoted by $\left. T_{\Omega}\right|_{W}$, is obtained by deleting columns which corresponding to $X\not\in W$.

\begin{dfn}\label{d3.1.2}
\begin{itemize}
\item[(i)] Assume $f$ is a subset antecedence/consequence solution with respect to $W$, and $g$ is a  subset antecedence/consequence solution with respect to $W'$. $f$ is said to superior to $g$, denoted by $f \succ g$ (or $g\prec f$), if $W'\subset W$ and
$$
f|_{W'}=g|_{W'}.
$$
\item[(ii)] Assume $f$ is a $W$-antecedence/consequence solution. $f$ is said to be a maximum antecedence/consequence solution, if $g$ is a $W'$-antecedence/consequence solution and $g\succ f$, then $W'=W$ and $g|_{W}=f|_{W}$.
\end{itemize}
\end{dfn}

Note that if a $W$-subset antecedence/consequence solution $f$ is a maximum antecedence/consequence solution, then $W$ is completely determined by $T$ (precisely speaking, by the non-zero columns of $T$), and is independent of $f$. Hence, $W$ is called a maximum (restricted) subset.

Using the same argument as in the proof of Theorem \ref{t2.2.3}, it is straightforward to prove the following result.

\begin{thm}\label{t3.1.3}
Assume the truth table of equations (\ref{2.2.1}) with respect to $\Omega\subset {\cal D}^s$ is $T_{\Omega}$, and state subset
$W\subset {\cal D}^n$ is given.
Then
\begin{itemize}
\item[(i)] (\ref{2.2.2}) is a subset antecedence solution of (\ref{2.2.1}) with respect to subset $W$, if and only if,
\begin{align}\label{3.1.2}
\left.M_G\right|_{W}\leq \left.T_{\Omega}\right|_{W}.
\end{align}
\item[(ii)] (\ref{2.2.2}) is a subset consequence solution of (\ref{2.2.1}) with respect to subset $W$, if and only if,
\begin{align}\label{3.1.3}
\left.T_{\Omega}\right|_{W}\leq \left.M_G\right|_{W}.
\end{align}
\end{itemize}
\end{thm}

A subset antecedence/consequence solution with respect to subset $W$ is briefly called  $W$-antecedence/consequence solution.

\begin{exa}\label{e3.1.4} Recall Example \ref{e2.2.2}.
\begin{itemize}
\item[(i)] Since the truth matrix of (\ref{2.2.5}) is (\ref{2.2.7}). According to Theorem \ref{t2.2.3}, it is obvious that (\ref{2.2.5}) has no antecedence solution.
\item[(ii)] Let $W=\{\d_4^1,\d_4^3,\d_4^4\}$, which is a maximum subset. Set
$u=\d_4[2,1,2,2]x$. According to Theorem \ref{t3.1.3}, this $u$ is a $W$-antecedence solution, which is a maximum subset antecedence solution.
\item[(iii)] Let $W'=\{\d_4^1,\d_4^3\}$. Then
$u=\d_4[2,1,2,1]x$ is a $W'$-antecedence solution. It is also clear that this $u$ is not a maximum solution.
\end{itemize}
\end{exa}

\section{Design of State Feedback Stabilizer}

\begin{dfn}\label{d3.2.1} Consider a BCN
\begin{align}
\label{3.2.1}
\begin{cases}
X_1(t+1)=F_1(X_1(t),\cdots,X_n(t),U_1(t),\cdots,U_m(t)),\\
X_2(t+1)=F_2(X_1(t),\cdots,X_n(t),U_1(t),\cdots,U_m(t)),\\
\vdots\\
X_n(t+1)=F_n(X_1(t),\cdots,X_n(t),U_1(t),\cdots,U_m(t)),\\
\end{cases}
\end{align}
where $X_i(t)\in {\cal D}$, $i=1,\cdots,n$ are state variables,   $U_j(t)\in {\cal D}$, $j=1,\cdots,m$ are controls, $F_i:{\cal D}^{m+n}\ra {\cal D}$, $i=1,\cdots,n$ are Boolean functions.

Let $X^d=(X^d_1,X^d_2,\cdots,X^d_n)\in {\cal D}^n$ be a given state. The BCN is said to be state feedback stabilizable to $X^d$, if there exists a state feedback control sequence
\begin{align}\label{3.2.2}
\begin{cases}
U_1(t)=G_1(X_1(t),\cdots,X_n(t)),\\
U_2(t)=G_2(X_1(t),\cdots,X_n(t)),\\
\vdots\\
U_m(t)=G_m(X_1(t),\cdots,X_n(t)),\\
\end{cases}
\end{align}
such that for the closed-loop network, there exists a $T \in \mathbb Z_{\geq 0}$ such that $X(t)=X^d$, for $t\geq T$.
\end{dfn}

\begin{dfn}\label{d3.2.2}
Consider BCN (\ref{3.2.1}). A state $X^d\in {\cal D}^n$ is called a control fixed point, if there exists a control sequence $U=(U_1, U_2,\cdots,U_m)$, such that
\begin{align}\label{3.2.3}
F_i(X^d,U)=X^d_i,\quad i=1,2,\cdots,n.
\end{align}
\end{dfn}

\begin{rem}\label{r3.2.201} It is obvious that (\ref{3.2.3}) is equivalent to: there exists a state feedback control sequence $U(X)$, such that
\begin{align}\label{3.2.4}
F_i(X^d,U(X^{d}))=X^d_i,\quad i=1,2,\cdots,n.
\end{align}

In fact, for a BCN, free control sequence is equivalent to state feedback control. It is obvious that state feedback control can be replaced by free control sequence. Conversely, since there are only finite states in a BCN, if a free control sequence can do something, then at each point $X\in {\cal D}^n$,  a control $U$ is selected, then this control can be described as $U(X)$.
\end{rem}

The following necessary condition is obvious.

\begin{prp}\label{p3.2.3} Assume BCN (\ref{3.2.1}) is stabilizable to $X^d$, then $X^d$ is a control fixed point.
\end{prp}

Next, we provide an algorithm, which provides a constructive procedure for stabilizer.

\begin{alg}\label{a3.2.4}
\begin{itemize}
\item Step 0 (Initial Step): Set initial object $\Omega(0)=W_0:=\{(C_1,C_2,\cdots C_n)\}$, where $X^d=(C_1,C_2,\cdots,C_n)$ is the destination point, to which the BCN is designed to be  stabilized. Construct $T_{\Omega(0)}$, and using it to find the maximum set $W_1$ with respect to $\Omega(0)$, and check if $X^d\in W_1$. If $X^d\not\in W_1$, $X^d$ is not a control fixed point, and the corresponding state feedback stabilization is not solvable. Stop the algorithm.

\item Step $k$ (Repeated Step, $k\geq 1$) : Set $\Omega(k)=W_{k}\backslash\bigcup\limits_{i=0}^{k-1} W_{i}$.
\begin{itemize}
\item[(i)] If $\Omega(k)=\emptyset$, then the problem is not solvable. Stop.
\item[(ii)] Else (i.e., $\Omega(k)\neq \emptyset$). Check if
\begin{align}\label{3.2.401}
\bigcup_{i=0}^k \Omega_{i}=\D_{2^n}.
\end{align}
If (\ref{3.2.401}) holds, the stabilization problem is solvable, set last $k$ as $k_*$ and go to Final Step.
Else (i.e., (\ref{3.2.401}) does not hold),  construct $T_{\Omega(k)}$, and use it to find the maximum set $W_{k+1}$ with respect to $\Omega(k)$. Then go back to Step $k+1$.
\end{itemize}

\item Final Step (Stabilizer Constructing Step): Construct $u(t)=M_Gx(t)$ by the following inequalities:
\begin{align}\label{3.2.5}
\left.M_G\right|_{\Omega(i)}
\begin{cases}
\leq \left.T_{\Omega_{0}}\right|_{\Omega(0)},\quad i=0,\\
\leq \left.T_{\Omega_{i-1}}\right|_{\Omega(i)},\quad i=1,\cdots,k^*.
\end{cases}
\end{align}

\end{itemize}

\end{alg}

\begin{thm}\label{t3.2.5}
BCN (\ref{3.2.1}) is stabilizable to $X^d$, if and only if, the Algorithm \ref{a3.2.4} can go through and provide a state feedback stabilizer.
\end{thm}

\begin{proof}
Starting from first step, one easily sees that $X\in \Omega(0)$ means there exists a $W_0$-antecedence solution, as a state feedback control, which assures $F_i(X,U(X))=C_i$, $i=1,2,\cdots,n$. That is, the state feedback control can drive $X\in \Omega(0)$ to $C$. Similar argument shows that at step k,
there is a state feedback control, when restrict it on $\Omega(k)$, it drives $X\in \Omega(k)$ to $\Omega(k-1)$. Eventually, if
there exists a $k^*$ such that (\ref{3.2.401}) holds, then all the points can be driven to $C$ within $k^*$ steps. So the successful of the algorithm means the state feedback stabilization is solvable.

To see that if the algorithm fails, the problem is not solvable, we consider $k^*$ as the last step for $\Omega_{k}\neq \emptyset$ and assume
$$
\bigcup_{i=0}^{k^*} \Omega_{i}\neq \D_{2^n}.
$$
Then, there exists a $X\in \D_{2^n}\backslash \left\{\bigcup_{i=0}^{k^*} \Omega_{i}\right\}$. It is clear that no state feedback can drive it to $C$.
\end{proof}

\begin{rem}\label{r3.2.501}
From the proof of Theorem \ref{t3.2.5} one sees easily that if the stabilization problem is solvable, then the state feedback stabilizers obtained from Algorithm \ref{a3.2.4} are ``optimal" in the sense that each point $X\in {\cal D}^n$ can reach the destination at the shortest time.
If we do not require ``time-optimal", we may have more state feedback stabilizers.
If in Algorithm \ref{a3.2.4} the $\Omega(i)$ are replaced by $W_i$, then in addition to time-optimal stabilizers, there are some non-time optimal stabilizers, unfortunately, there are also some state feedback controls, which can not stabilize the network.
\end{rem}

We use an example to depict this algorithm.

\begin{exa}\label{e3.2.6}

Consider the following BCN:

\begin{align}\label{100}
\begin{cases}
X_1(t+1)=X_2(t)\vee U_1(t),\\
X_2(t+1)=X_4(t)\vee\left(U_2(t)\wedge X_1(t)\right),\\
X_3(t+1)=\left(X_1(t)\wedge X_4(t)\right)\bar{\vee}\left(\neg X_3(t)\right),\\
X_4(t+1)=\left(\neg X_1(t)\right)\lra U_2(t).
\end{cases}
\end{align}
Is it possible to stabilize it to $X^d=(1,1,0,1)$?

Consider the following equation:
\begin{align}\label{101}
\begin{cases}
F_1(X,U)=X_2\vee U_1,\\
F_2(X,U)=X_4\vee\left(U_2\wedge X_1\right),\\
F_3(X,U)=\left(X_1\wedge X_4\right)\bar{\vee}\left(\neg X_3\right),\\
F_4(X,U)=\left(\neg X_1\right)\lra U_2.
\end{cases}
\end{align}
Let the structure matrix of $F_i$ be $M_i$, $i=1,2,3,4$. Then $F:=F_1\ltimes F_2\ltimes F_3\ltimes F_4$ has its structure matrix
\begin{align}\label{102}
\begin{array}{ccl}
M_F&=&M_1*M_2*M_3*M_4\\
~&=&\d_{16}[ 2, 4, 4, 2, 2, 4, 4, 2, 3, 7, 1, 5, 3, 7, 1, 5,\\
~&~&~~~~     1, 7, 3, 5, 1, 7, 3, 5, 4, 8, 2, 6, 4, 8, 2, 6,\\
~&~&~~~~     2, 4, 4, 2,10,12,12,10, 3, 7, 1, 5,11,15, 9,13,\\
~&~&~~~~     1, 7, 3, 5, 9,15,11,13, 4, 8, 2, 6,12,16,10,14].
\end{array}
\end{align}
Then
$$
\Omega(0)=W_0=\delta^3_{16} \sim \{C=(1,1,0,1)\}.
$$
Then the truth  matrix of $\Omega(0)$ can be obtained as
\begin{align}\label{103}
T_{\Omega(0)} =\left[
\begin{array}{cccccccccccccccc}
0&0&0&0&0&0&0&0&1&0&0&0&1&0&0&0\\
0&0&1&0&0&0&1&0&0&0&0&0&0&0&0&0\\
0&0&0&0&0&0&0&0&1&0&0&0&0&0&0&0\\
0&0&1&0&0&0&0&0&0&0&0&0&0&0&0&0\\
\end{array}\right].
\end{align}
From (\ref{103}) it is clear that
\begin{align}\label{104}
W_1=\d_{16}\{3,7,9,13\}.
\end{align}
Since $(1,1,0,1)\sim \d_{16}^3\in W_1$, $(1,1,0,1)$, the objective state is a control fixed point.

Next, set
\begin{align}\label{105}
\Omega(1)=W_1\backslash W_0=\d_{16}\{7,9,13\}.
\end{align}
We consider
\begin{align}\label{106}
(F_1(X,U), F_2(X,U), F_3(X,U), F_4(X,U))\in \Omega(1),
\end{align}
The truth matrix can be obtained as
\begin{align}\label{107}
T_{\Omega(1)} =\left[
\begin{array}{cccccccccccccccc}
0&0&0&0&0&0&0&0&0&1&0&0&0&1&0&0\\
0&1&0&0&0&1&0&0&0&0&0&0&0&0&0&0\\
0&0&0&0&0&0&0&0&0&1&0&0&0&0&1&1\\
0&1&0&0&1&0&0&1&0&0&0&0&0&0&0&0\\
\end{array}\right].
\end{align}
Then we have
\begin{align}\label{108}
W_2=\d_{16}\{2,5,6,8,10,14,15,16\}.
\end{align}
\begin{align}\label{109}
\begin{array}{ccl}
\Omega(2)&=&W_2\backslash \{W_1\cup W_0\}\\
~&=&\d_{16}\{2,5,6,8,10,14,15,16\}.
\end{array}
\end{align}
Third step, we consider
\begin{align}\label{110}
(F_1(X,U), F_2(X,U), F_3(X,U), F_4(X,U))\in \Omega(2),
\end{align}
\begin{align}\label{111}
T_{\Omega(2)} =\left[
\begin{array}{cccccccccccccccc}
1&0&0&1&1&0&0&1&0&0&0&1&0&0&0&1\\
0&0&0&1&0&0&0&1&0&1&1&1&0&1&1&1\\
1&0&0&1&1&0&0&1&0&0&0&1&0&1&0&0\\
0&0&0&1&0&1&0&0&0&1&1&1&0&1&1&1\\
\end{array}\right].
\end{align}
Then we have
\begin{align}\label{112}
W_3=\d_{16}\{1,4,5,6,8,10,11,12,14,15,16\}.
\end{align}
Hence,
\begin{align}\label{113}
\Omega(3):=W_3\backslash \{W_2\cup W_1\cup W_0\}=\d_{16}\{1,4,11,12\}.
\end{align}

Now since
$$
\Omega(0)\cup \Omega(1)\cup \Omega(2)\cup \Omega(3)=\D_{16},
$$
the BCN (\ref{100}) is state feedback stabilizable to $(1,1,0,1)$.

Finally, we construct the state feedback control. It should satisfy the following condition
\begin{align}\label{114}
\begin{array}{lr}
U|_{\Omega(1)\cup \Omega(0)}\leq T_{\Omega(0)}|_{\Omega(1)\cup \Omega(0)}&~~~~~(a)\\
U|_{\Omega(2)}\leq T_{\Omega(1)}|_{\Omega(2)}&~~~~~(b)\\
U|_{\Omega(3)}\leq T_{\Omega(2)}|_{\Omega(3)}&~~~~~(c)\\
\end{array}
\end{align}
For (\ref{114})(a), one feasible choice is:
\begin{align}\label{115}
\begin{array}{l}
u(\d_{16}^3)=\d_4^4\in \d_4\{2,4\},\\
u(\d_{16}^7)=\d_4^2,\\
u(\d_{16}^9)=\d_4^3\in \d_4\{1,3\},\\
u(\d_{16}^{13})=\d_4^1.\\
\end{array}
\end{align}
Note that where $\in\d_4\{*,*\}$ means it has multiple choices.

For (\ref{114})(b), one feasible choice is:
\begin{align}\label{116}
\begin{array}{l}
u(\d_{16}^2)=\d_4^2\in \d_4\{2,4\},\\
u(\d_{16}^5)=\d_4^4,\\
u(\d_{16}^6)=\d_4^2,\\
u(\d_{16}^{8})=\d_4^4,\\
u(\d_{16}^{10})=\d_4^3\in \d_4\{1,3\},\\
u(\d_{16}^{14})=\d_4^1,\\
u(\d_{16}^{15})=\d_4^3,\\
u(\d_{16}^{16})=\d_4^3.\\
\end{array}
\end{align}
For (\ref{114})(c), one feasible choice is:
\begin{align}\label{117}
\begin{array}{l}
u(\d_{16}^1)=\d_4^1\in \d_4\{1,3\},\\
u(\d_{16}^4)=\d_4^2\in \d_4\{1,2,3,4\},\\
u(\d_{16}^{11})=\d_4^2\in \d_4\{2,4\},\\
u(\d_{16}^{12})=\d_4^4\in \d_4\{1,2,3,4\}.\\
\end{array}
\end{align}
Putting (\ref{115})-(\ref{117}) together, we have a state feedback stabilizer as
\begin{align}\label{118}
u(t)=M_Gx(t),
\end{align}
where
$$
M_G=\d_{4}[1,2,4,2,4,2,2,4,3,3,2,4,1,1,3,3].
$$
Then
$$
\begin{array}{ccl}
M_G^1&=&(I_2\otimes \J^T_2)M_G\\
~&=&\d_2[1,1,2,1,2,1,1,2,2,2,1,2,1,1,2,2],
\end{array}
$$
$$
\begin{array}{ccl}
M_G^2&=&(\J^T_2\otimes I_2)M_G\\
~&=&\d_2[1,2,2,2,2,2,2,2,1,1,2,2,1,1,1,1].
\end{array}
$$
Then it is easy to figure out the state feedback stabilizer as
\begin{align}\label{119}
\begin{array}{ccl}
U_1(t)&=&\left(X_1(t)\wedge X_2(t)\wedge X_3(t)\right)\\
~&~&\vee  \left(X_1(t)\wedge X_2(t)\wedge \neg X_3(t)\wedge \neg X_4(t)\right)\\
~&~&\vee  \left(X_1(t)\wedge \neg X_2(t)\wedge X_3(t)\wedge \neg X_4(t)\right)\\
~&~&\vee  \left(X_1(t)\wedge \neg X_2(t)\wedge \neg X_3(t)\wedge  X_4(t)\right)\\
~&~&\vee  \left(\neg X_1(t)\wedge X_2(t)\wedge \neg X_3(t)\wedge X_4(t)\right)\\
~&~& \vee  \left(\neg X_1(t)\wedge \neg X_2(t)\wedge X_3(t)\right),\\
U_2(t)&=&\left(X_1(t)\wedge X_2(t)\wedge X_3(t)\wedge X_4(t)\right)\\
~&~&\vee  \left(\neg X_1(t)\wedge X_2(t)\wedge X_3(t)\right)\\
~&~&\vee \left(\neg X_1(t)\wedge \neg X_2(t)\right).\\
\end{array}
\end{align}

To verify whether this stabilizer works we calculate the closed-loop BCN, which is
\begin{align}\label{120}
\begin{array}{ccl}
x(t+1)&=&M_Fu(t)x(t)=M_FM_Gx(t)x(t)\\
~&=&M_FM_GPR_{16}x(t):=M_cx(t),
\end{array}
\end{align}
where
$$
\begin{array}{ccl}
M_c&=&M_FM_G PR_{16}\\
~&=&\d_{16}[2,7,3,5,9,7,3,13,3,7,2,6,3,7,9,13].
\end{array}
$$

It is ready to verify that
$$
M_c^3=\d_{16}[3,3,3,3,3,3,3,3,3,3,3,3,3,3,3,3].
$$
That is, after three steps all the trajectories of the closed-loop BCN converge to $\d_{16}^3\sim (1,1,0,1)$.

The state transition graph of (\ref{120}) is shown in Fig. \ref{Fig3.2.1}, where the only attractor is the point $\d_{16}^3\sim (1,1,0,1)$.


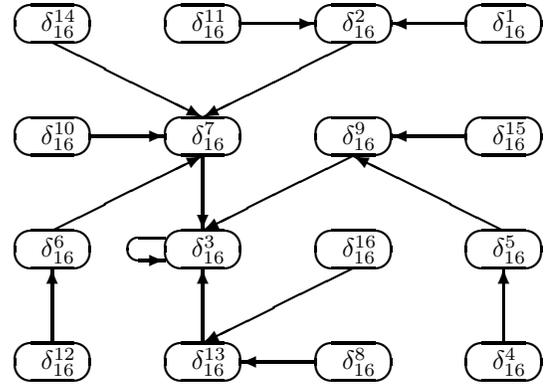
\begin{figure}

\centering
\setlength{\unitlength}{0.5cm}
\begin{picture}(16,12)\thicklines
\put(2,1.5){\oval(2,1)}
\put(6,1.5){\oval(2,1)}
\put(10,1.5){\oval(2,1)}
\put(14,1.5){\oval(2,1)}
\put(2,4.5){\oval(2,1)}
\put(6,4.5){\oval(2,1)}
\put(10,4.5){\oval(2,1)}
\put(14,4.5){\oval(2,1)}
\put(2,7.5){\oval(2,1)}
\put(6,7.5){\oval(2,1)}
\put(10,7.5){\oval(2,1)}
\put(14,7.5){\oval(2,1)}
\put(2,10.5){\oval(2,1)}
\put(6,10.5){\oval(2,1)}
\put(10,10.5){\oval(2,1)}
\put(14,10.5){\oval(2,1)}
\put(2,2){\vector(0,1){2}}
\put(6,2){\vector(0,1){2}}
\put(10,4){\vector(-2,-1){4}}
\put(9,1.5){\vector(-1,0){2}}
\put(14,2){\vector(0,1){2}}
\put(2,5){\vector(2,1){4}}
\put(6,7){\vector(0,-1){2}}
\put(10,7){\vector(-2,-1){4}}
\put(14,5){\vector(-2,1){4}}
\put(3,7.5){\vector(1,0){2}}
\put(13,7.5){\vector(-1,0){2}}
\put(2,10){\vector(2,-1){4}}
\put(10,10){\vector(-2,-1){4}}
\put(7,10.5){\vector(1,0){2}}
\put(13,10.5){\vector(-1,0){2}}
\put(5,4.5){\oval(2,0.6)[l]}
\put(4.5,4.2){\vector(1,0){0.5}}
\put(1.7,1.3){$\d_{16}^{12}$}
\put(5.7,1.3){$\d_{16}^{13}$}
\put(9.7,1.3){$\d_{16}^8$}
\put(13.7,1.3){$\d_{16}^4$}
\put(1.7,4.3){$\d_{16}^{6}$}
\put(5.7,4.3){$\d_{16}^{3}$}
\put(9.7,4.3){$\d_{16}^{16}$}
\put(13.7,4.3){$\d_{16}^5$}
\put(1.7,7.3){$\d_{16}^{10}$}
\put(5.7,7.3){$\d_{16}^{7}$}
\put(9.7,7.3){$\d_{16}^9$}
\put(13.7,7.3){$\d_{16}^{15}$}
\put(1.7,10.3){$\d_{16}^{14}$}
\put(5.7,10.3){$\d_{16}^{11}$}
\put(9.7,10.3){$\d_{16}^2$}
\put(13.7,10.3){$\d_{16}^1$}
\end{picture}

\caption{State Transition Graph of (\ref{120})\label{Fig3.2.1}}

\end{figure}
Finally, it is worthy noting that from (\ref{115})-(\ref{117}) one sees easily that there are
$$
2\times 2\times 2\times 2\times 2\times 4\times 2\times 4=1024
$$
time optimal feasible state feedback stabilizers.
\end{exa}

\section{Set Stabilization}

\begin{dfn}\label{d4.0} Consider BCN (\ref{3.2.1}).
Let $\cal M \subset {\cal D}^n$ be a given set of states. The BCN is said to be state feedback stabilizable to $\cal M$, if there exists a state feedback control sequence (\ref{3.2.2}) such that for any initial state $X_0$ of the closed-loop network, there exists a $T \in \mathbb Z_{\geq 0}$ such that,
\begin{align}\label{4.0}
X(t; X_0,U(X))\in \cal M, \quad \forall t \geq T.
\end{align}

\end{dfn}

\begin{dfn}\label{d4.1}
Consider BCN (\ref{3.2.1}), and assume $\cal M\subset {\cal D}^n$. $\cal M$ is said to be control invariant set, if for each state $X_0\in \cal M$, there exists a control sequence $U=(U_1,U_2,\cdots,U_m)$ such that $F(X_0,U)\in \cal M$.
\end{dfn}

To verify whether a state set $\cal M$ is a control invariant set, we have the following results.

\begin{prp}\label{p4.2} Consider BCN  (\ref{3.2.1}) and assume a set of states $\cal M\subset \cal D^n$ is given to be stabilized. Construct the corresponding truth matrix $T_{\cal M}$. Let the maximum set for $T_{\cal M}$ be $\Theta$. Then $\cal M$ is a control invariant set, if,
\begin{align}\label{4.5}
\cal M\cap \Theta =\cal M
\end{align}
In other words, $\cal M \subset \Theta$.
\end{prp}
\begin{proof}
By definition, there exists at least one $\Theta$-antecedence solution $U(X)$ such that
\begin{align}\label{4.6}
F(X_0,U)\in \cal M, \quad \forall X_0 \in \Theta
\end{align}

Since $\cal M\subset \Theta$, then  it is obvious that $F(X_0,U)\in \cal M,\quad X_0\in \cal M$.
\end{proof}

On the other hand, in more general scenarios where $\cal M \not\subset \Theta$, yet $\cal M \cap \Theta \neq \emptyset$, we have the following remark.

\begin{rem}\label{p4.3} In the aforementioned scenarios where $\cal M \not\subset \Theta$, yet $\cal M \cap \Theta \neq \emptyset$, the control invariant set $W_0$ within the state set $\cal M$ can be determined by,
\begin{align}\label{4.7}
W_0 = \cal M \cap \Theta
\end{align}
\end{rem}

The proof is the same as above, and for brevity, it is not demonstrated here. In \cite {guo15}, the concept of the largest control invariant set is given, through which, the optimal state feedback controls can be designed. Based on the property of Ledley Antecedence Solution, it is quite obvious that the control invariant set solved using the maximum set of the truth matrix and (\ref{4.7}) is the largest control invariant set. Similar to point stabilization case, we have the following algorithm.

\begin{alg}\label{a4.4}
\begin{itemize}
\item Step 0 (determining the largest control invariant set): Set the destination set $\cal M$, to which the BCN is designed to be stabilized. Construct $T_{\cal M}$, and use it to find the maximum set $\Theta$ with respect to $\cal M$. Check if $\cal M \cap \Theta = \emptyset$. If $\cal M \cap \Theta = \emptyset$, the corresponding state feedback stabilization is not solvable. Stop the algorithm. Else (i.e., $W_0 = \cal M \cap \Theta$), calculate the corresponding truth matrix $T_{W_0}$ and the maximum set $W_1$, set $\Omega(1)=W_1\backslash W_0$ and go to step $k=1$.

\item Step $k$ (Repeated Step, $k\geq 1$) :  Same as Step $k$ of Algorithm \ref{a3.2.4}.

\item Final Step (Stabilizer Constructing Step): Same as Final Step of Algorithm \ref{a3.2.4}.

\end{itemize}

\end{alg}

The following result can be proved using similar argument for Theorem \ref{t3.2.5}.
\begin{thm}\label{t4.5}
BCN (\ref{3.2.1}) is stabilizable to $W_0$, if and only if, the Algorithm \ref{a4.4} can go through and provide a state feedback stabilizer.
\end{thm}

It is also worthy noting that this algorithm provides the time-optimal stabilizers. 
In the following an example is used to describe the design process.

\begin{exa}\label{e4.6} Recall Example \ref{e3.2.6}, and consider BCN (\ref{100}), and set $\cal M=\{(1,0,1,0)\sim \d_{16}^6, (1,0,0,1)\sim \d_{16}^7,
(0,1,0,1)\sim \d_{16}^{12}\}$. This example investigates whether BCN (\ref{100}) can be stabilized to $\cal M$.

First, construct the truth matrix $T_{\cal M}$ corresponding to $\cal M$,
\begin{align*}
T_{\cal M} =\left[
\begin{array}{cccccccccccccccc}
0&0&0&0&0&0&0&0&0&1&0&0&0&1&0&0\\
0&1&0&0&0&1&0&0&0&0&0&1&0&0&0&1\\
0&0&0&0&0&1&1&0&0&1&0&0&0&0&0&0\\
0&1&0&0&0&0&0&0&0&0&0&1&1&0&0&0
\end{array}\right].
\end{align*}

The maximum set is
$$
\Theta=\d_{16}\left\{2,6,7,10,12,13,14,16\right\}.
$$
Since $\cal M \cap \Theta = \cal M$, set $\cal M$ is control invariant set. For easier understanding, denote $W_0=\cal M$, $W_1=\Theta$, set $\Omega(1):=W_1\backslash W_0$. Then it is easy to calculate that
\begin{align*}
T_{\Omega(1)} =\left[
\begin{array}{cccccccccccccccc}
1&0&0&1&1&0&0&1&0&0&0&0&0&0&0&0\\
0&0&0&0&0&0&0&0&0&0&1&0&0&0&1&0\\
1&0&0&1&1&0&0&1&0&0&0&0&0&0&0&1\\
0&0&0&0&0&0&0&1&0&0&1&0&0&1&1&1
\end{array}\right].
\end{align*}

Then $W_2=\d_{16}\{1,4,5,8,11,14,15,16\}$, and
set
$$
\begin{array}{ccl}
\Omega(2)&:=&W_2\backslash (W_1\cup W_0)\\
~&=& \d_{16}\{1,4,5,8,11,15\}.
\end{array}
$$
We calculate that
\begin{align*}
T_{\Omega(2)} =\left[
\begin{array}{cccccccccccccccc}
0&1&1&0&0&1&1&0&0&0&1&1&0&0&1&1\\
1&0&0&1&1&0&0&1&1&1&0&0&1&1&0&0\\
0&1&1&0&0&0&0&0&0&0&1&1&1&1&0&0\\
1&0&0&1&0&1&1&0&1&1&0&0&0&0&0&0
\end{array}\right].
\end{align*}

One sees that
$$
\begin{array}{ccl}
W_3&=\d_{16}\{&1,2,3,4,5,6,7,8,\\
~&~&9,10,11,12,13,14,15,16\},
\end{array}
$$
and
$$
\begin{array}{ccl}
\Omega(3)&=W_3\backslash (W_2\cup W_1\cup W_0)=\d_{16}\{3,9\}.
\end{array}
$$
Since
$$
\Omega(0)\cup \Omega(1)\cup \Omega(2)\cup \Omega(2)=\D_{16},
$$
we conclude that the BCN (\ref{100}) is a state feedback stabilizable to $\cal M$.

Afterwards, we construct the state feedback stabilizer.

Since $U(X)$ for $X\in \Omega(0) \cup \Omega(1)$ is determined by $T_{\Omega(0)}$, we choose one feasible solution as follows:

\begin{align}\label{4.2}
\begin{array}{l}
u(\d_{16}^2)=\d_4^4\in \d_4\{2,4\}\\
u(\d_{16}^6)=\d_4^3\in \d_4\{2,3\}\\
u(\d_{16}^7)=\d_4^3\\
u(\d_{16}^{10})=\d_4^1\in \d_4\{1,3\}\\
u(\d_{16}^{12})=\d_4^4\in \d_4\{2,4\}\\
u(\d_{16}^{14})=\d_4^1\\
u(\d_{16}^{16})=\d_4^2.
\end{array}
\end{align}

Using $T_{\Omega(1)}$, we can determine  $U(X)$ for $X\in \Omega(2)$, that is

\begin{align}\label{4.3}
\begin{array}{l}
u(\d_{16}^1)=\d_4^1\in \d_4\{1,3\}\\
u(\d_{16}^4)=\d_4^1\in \d_4\{1,3\}\\
u(\d_{16}^5)=\d_4^1\in \d_4\{1,3\}\\
u(\d_{16}^{8})=\d_4^4\in \d_4\{1,3,4\}\\
u(\d_{16}^{11})=\d_4^2\in \d_4\{2,4\}\\
u(\d_{16}^{15})=\d_4^4\in \d_4\{2,4\}.
\end{array}
\end{align}

Using $T_{\Omega(2)}$, we can determine  $U(X)$ for $X\in \Omega(3)$ as

\begin{align}\label{4.4}
\begin{array}{l}
u(\d_{16}^3)=\d_4^3\in \d_4\{1,3\}\\
u(\d_{16}^9)=\d_4^2\in \d_4\{2,4\}.
\end{array}
\end{align}

Assume
$$
u(t)=M_Gx(t),
$$
where $M_G\in {\cal L}_{4\times 16}$. Summarizing the choices (\ref{4.2})-(\ref{4.4}), a feasible control is
\begin{align}\label{200}
M_G=\d_{4}[1,2,3,1,1,3,3,4,2,1,2,4,4,1,4,2].
\end{align}
We conclude that
(\ref{200}) provides a state feedback stabilizer for BCN (\ref{100}).

Finally, we may check the closed-loop system. It is
\begin{align}\label{201}
\begin{array}{ccl}
x(t+1)&=&M_Fu(t)x(t)=M_FM_Gx(t)x(t)\\
~&=&M_FM_GPR_{16}x(t):=M_Cx(t).
\end{array}
\end{align}
It is easy to calculate that
\begin{align}\label{202}
M_C=\d_{16}
[2,7,4,2,2,12,12,13,4,7,2,6,12,7,10,6].
\end{align}

The state transition graph of this closed-loop network is described in Fig. \ref{Fig4.1}, which shows that there are only one attractor $C=\{\d_{12}^6,~\d_{12}^{12}\}$. Hence, all trajectories converge to $C\subset W_0$.

\vskip 2mm

\begin{figure}
\centering
\setlength{\unitlength}{0.5cm}
\begin{picture}(16,16)\thicklines
\put(10,1.5){\oval(2,1)}
\put(14,1.5){\oval(2,1)}
\put(2,4.5){\oval(2,1)}
\put(6,4.5){\oval(2,1)}
\put(10,4.5){\oval(2,1)}
\put(14,4.5){\oval(2,1)}
\put(2,7.5){\oval(2,1)}
\put(6,7.5){\oval(2,1)}
\put(10,7.5){\oval(2,1)}
\put(14,7.5){\oval(2,1)}
\put(14,10.5){\oval(2,1)}
\put(14,13.5){\oval(2,1)}
\put(4.5,10.5){\oval(2,1)}
\put(7.5,10.5){\oval(2,1)}
\put(10,10.5){\oval(2,1)}
\put(2,10.5){\oval(2,1)}
\put(2,5){\vector(0,1){2}}
\put(2,10){\vector(0,-1){2}}
\put(6,5){\vector(0,1){2}}
\put(4,10){\vector(1,-1){2}}
\put(8,10){\vector(-1,-1){2}}
\put(10,2){\vector(0,1){2}}
\put(10,5){\vector(0,1){2}}
\put(10,10){\vector(0,-1){2}}
\put(14,2){\vector(0,1){2}}
\put(14,5){\vector(0,1){2}}
\put(14,13){\vector(0,-1){2}}
\put(3,7.5){\vector(1,0){2}}
\put(7,7.5){\vector(1,0){2}}
\put(11,7.5){\vector(1,0){2}}
\put(13.5,8){\vector(0,1){2}}
\put(14.5,10){\vector(0,-1){2}}
\put(1.7,4.3){$\d_{16}^9$}
\put(1.7,7.3){$\d_{16}^4$}
\put(1.7,10.3){$\d_{16}^3$}
\put(4.2,10.3){$\d_{16}^5$}
\put(7.2,10.3){$\d_{16}^1$}
\put(5.7,4.3){$\d_{16}^{11}$}
\put(5.7,7.3){$\d_{16}^2$}
\put(9.7,1.3){$\d_{16}^{15}$}
\put(9.7,4.3){$\d_{16}^{10}$}
\put(9.7,7.3){$\d_{16}^7$}
\put(9.7,10.3){$\d_{16}^{14}$}
\put(13.7,1.3){$\d_{16}^8$}
\put(13.7,4.3){$\d_{16}^{13}$}
\put(13.7,7.3){$\d_{16}^{12}$}
\put(13.7,10.3){$\d_{16}^6$}
\put(13.7,13.3){$\d_{16}^{16}$}
\end{picture}
\caption{State Transition Graph of (\ref{201})}\label{Fig4.1}
\end{figure}
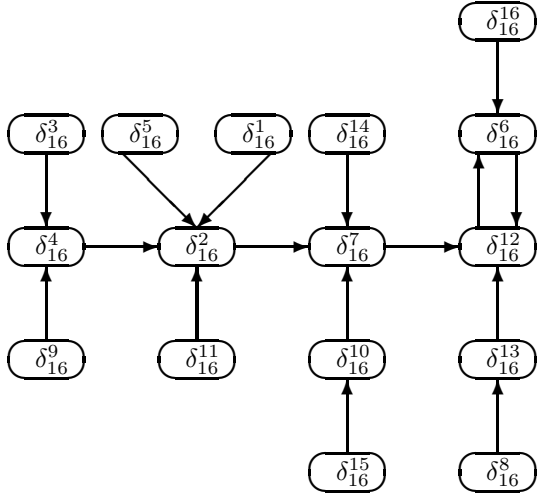

\vskip 2mm

From (\ref{4.2})-(\ref{4.4}) it is easy to see that there are 6144 feasible state feedback stabilizers.

\end{exa}

\section{Stabilization of General Logical Systems}

The method discussed in previous sections are all applicable to $k$-valued or  mix-valued logic. In this section we give an example to demonstrate this.

\begin{exa}\label{est.6.1} Consider a mix-valued logical network
\begin{align}\label{st.6.1}
\begin{cases}
X_1(t+1)=X_1(t)\diamondsuit X_2(t),\\
X_2(t+1)=X_1(t)\Box U(t),
\end{cases}
\end{align}
where $X_1(t)\in {\cal D}_2$, $X_2(t),~U(t)\in {\cal D}_3$, $\diamondsuit :{\cal D}_2\times {\cal D}_3\ra {\cal D}_2$, $\Box:  {\cal D}_2\times {\cal D}_3\ra {\cal D}_3$, with their structure matrices as

\begin{align}\label{st.6.2}
M_{\diamondsuit}=\d_2[1,1,2,2,2,1];
\end{align}

\begin{align}\label{st.6.3}
M_{\Box}=\d_3[1,2,3,2,3,1].
\end{align}

\begin{itemize}
\item[(i)] Is it possible to stabilize (\ref{st.6.1}) to $(1,1)\sim\d_6^1$?

It is easy to calculate that
\begin{align}\label{st.6.301}
x(t+1)=M_Fu(t)x(t),
\end{align}
where
$$
M_F=\d_6[1,1,4,5,5,2,2,2,5,6,6,3,3,3,6,4,4,1].
$$

Set $\Omega(0)=W_0=\{(1,1)\}$. Using (\ref{st.6.301}), the truth matrix with respect to $\Omega(0)$ is
\begin{align}\label{st.6.4}
T_{\Omega(0)}=\begin{bmatrix}
1&1&0&0&0&0\\
0&0&0&0&0&0\\
0&0&0&0&0&1\\
\end{bmatrix}.
\end{align}
Hence,
\begin{align}\label{st.6.5}
W_1=\d_6\{1,2,6\},
\end{align}
and
\begin{align}\label{st.6.6}
\Omega(1)=W_1\backslash W_0=\d_6\{2,6\}.
\end{align}

The truth matrix with respect to $\Omega(1)$ is
\begin{align}\label{st.6.7}
T_{\Omega(1)}=\begin{bmatrix}
0&0&0&0&0&1\\
1&1&0&1&1&0\\
0&0&1&0&0&0\\
\end{bmatrix}.
\end{align}
Hence,
\begin{align}\label{st.6.8}
W_2=\d_6\{1,2,3,4,5,6\},
\end{align}
and
\begin{align}\label{st.6.9}
\Omega(2)=W_2\backslash (W_0\cup W_1)=\d_6\{3,4,5\}.
\end{align}
Now
$$
\Omega(0)\cup \Omega(1)\cup \Omega(2)={\cal D}_6.
$$
Hence, mix-valued logical system (\ref{st.6.1}) can be stabilized to $(1,1)\sim \d_6^1$.

Next, we construct a state feedback stabilizer. Similarly to Boolean cases, a feasible state feedback control can be obtained as
\begin{align}\label{st.6.10}
u(t)=M_Gx(t),
\end{align}
where
$$
M_G=\d_3[1,1,3,2,2,3].
$$

To verify the result we consider the closed-loop network, which is
\begin{align}\label{st.6.11}
x(t+1)=M_FM_Gx^2(t)=M_FM_GPR_6x(t):=M_cx(t).
\end{align}
Then
$$
M_c=M_FM_GPR_6=\d_6[1,1,6,6,6,1].
$$
Since
$$
M_c^2=\d_6[1,1,1,1,1,1],
$$
one sees easily that after two steps the system stabilized to $\d_6^1$.

\item[(ii)]  Is it possible to stabilize (\ref{st.6.1}) to $\cal M=\{(1,0), (0,0)\}$?
To verify this, we first calculate the truth matrix corresponding to state set $\cal M$,
\begin{align}\label{st.6.12}
T_{\cal M}=\begin{bmatrix}
0&0&0&0&0&0\\
0&0&0&1&1&1\\
1&1&1&0&0&0\\
\end{bmatrix}.
\end{align}

where the maximum set $\Theta=\Delta_6$, and $\cal M \cap \Theta = \cal M$, therefore, the state set $\cal M$ is a control invariant set. From the truth table \ref{st.6.12}, one can easily sees that all states of the logic system can be stabilized to the state set $\cal M$ within one step, and the control sequence can be obtained as,

\begin{align}\label{st.6.13}
u(t)=\d_3[3,3,3,2,2,2]x(t)
\end{align}

Put this state feedback control into the original system, we have
the closed-loop system as
\begin{align}\label{st.6.14}
x(t+1)=M_FM_GPR_6x(t)=M_cx(t),
\end{align}
where
$$
M_c=M_FM_GPR_6=\d_6[3,3,6,6,6,3].
$$

Fig. \ref{Figst.6.1} is the state transition graph of the closed-loop network (\ref{st.6.14}). It shows the convergence of this system to $\cal M$.

\vskip 2mm

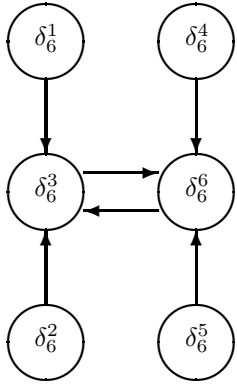
\begin{figure}
\centering
\setlength{\unitlength}{0.5cm}
\begin{picture}(8,12)\thicklines
\put(2,2){\oval(2,2)}
\put(6,2){\oval(2,2)}
\put(2,6){\oval(2,2)}
\put(6,6){\oval(2,2)}
\put(2,10){\oval(2,2)}
\put(6,10){\oval(2,2)}
\put(2,3){\vector(0,1){2}}
\put(6,3){\vector(0,1){2}}
\put(2,9){\vector(0,-1){2}}
\put(6,9){\vector(0,-1){2}}
\put(5,5.5){\vector(-1,0){2}}
\put(3,6.5){\vector(1,0){2}}
\put(1.7,1.9){$\d_{6}^{2}$}
\put(1.7,5.9){$\d_{6}^{3}$}
\put(1.7,9.9){$\d_{6}^1$}
\put(5.7,1.9){$\d_{6}^{5}$}
\put(5.7,5.9){$\d_{6}^{6}$}
\put(5.7,9.9){$\d_{6}^{4}$}
\end{picture}
\caption{State Transition Graph of (\ref{st.6.14})\label{Figst.6.1}}
\end{figure}

\vskip 2mm

\end{itemize}
\end{exa}

\section{Conclusion}

The Ledley antecedence/consequence solution has been applied for the design of feedback stabilizers in generic logic systems in this paper. To achieve such goal, the original theory has been extended in two ways. First, logical functions have been considered as an admissible set of values instead of a set of equations. Second, the domain of arguments has also been applied on a restricted subset of state space instead of the whole state space. Based on such knowledge, this paper has depicted that by properly designing admissible subsets for the logic functions of the systems, the state feedback controls can be obtained automatically when solving the subset antecedence solutions with restricted subsets for arguments. By using this approach, the stabilization of BCN to a pre-assigned point and to a pre-assigned set can be achieved respectively. An iterative algorithm has been designed to verify if the problem is solvable or not, and if the problem is solvable, the algorithm is also capable to provide all (time-optimal) state feedback stabilizers. Various examples have been demonstrated to verify the effectiveness of the proposed technique in Boolean control networks, as well as $k$-valued and mix-valued logical networks.

The technique introduced in this paper is new and useful. It shows that state feedback controls can be obtained via properly designed admissible sets and restricted sets. The method provided in this paper has a universal character, which is via properly designed sequence of subsets, the state feedback control may be obtained for various purposes (i.e., state feedback control, output feedback stabilization, tracking).

 \begin{IEEEbiography}[{\includegraphics[width=1in,height=1.25in,clip,keepaspectratio]{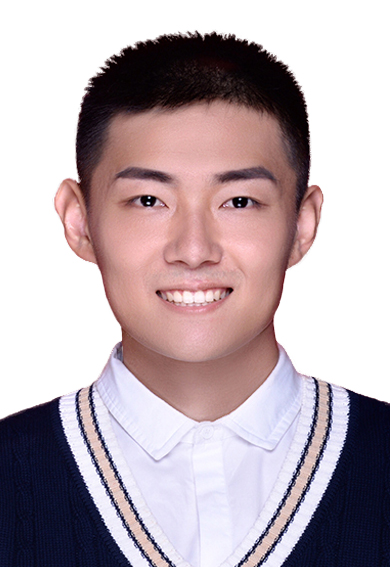}}]{Yingzhe Jia} (S'18) received the B.E. in electrical engineering from Shandong University, China, in 2015, and the Ph.D. degree in electrical and electronic engineering from the University of Manchester, Manchester, U.K. in 2020. He is currently a research associate with the school of Mathematics, Shandong University, China. His research interests include stabilization of logic control networks, application of game theory in power systems, etc.

\end{IEEEbiography}

\begin{IEEEbiography}[{\includegraphics[width=1in,height=1.25in,clip,keepaspectratio]{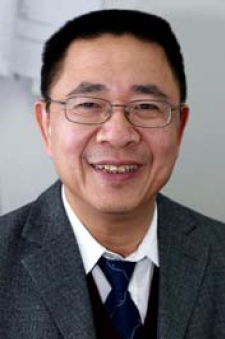}}]{Daizhan Cheng} (F'05) received the Bachelor¡¯s degree in mechanical engineering from Tsinghua University, Beijing, China, the M.S. degree in mathematics from the Graduate School, Chinese Academy of Sciences, Beijing, China, and the Ph.D. degree in systems science and mathematics from Washington University, St. Louis, WA, USA, in 1970, 1981, and 1985, respectively.

Since 1990 he has been a Professor both with the Institute of Systems Science, National University of Singapore, Singapore, and Academy of Mathematics and Systems Science, Chinese Academy of Sciences, Beijing, China. He has authored or coauthored 12 books, more than 240 journal papers, and more than 130 conference papers. His current research interests include nonlinear control systems, switched systems, Hamiltonian systems, Boolean control networks, and game theory.

Dr. Cheng was the recipient of the Second Grade National Natural Science Award of China in 2008 and 2014, and the Automatica 2008 to 2010 Best Theory/Methodology Paper Award, bestowed by IFAC, in 2011. He was the Chairman of the Technical Committee on Control Theory, Chinese Association of Automation from 2003 to 2010, and a Member of the IEEE CSS Board of Governors in 2009. He was an IFAC Fellow in 2008 and an IFAC Council Member from 2011 to 2014.
\end{IEEEbiography}

\begin{IEEEbiography}[{\includegraphics[width=1in,height=1.25in,clip,keepaspectratio]{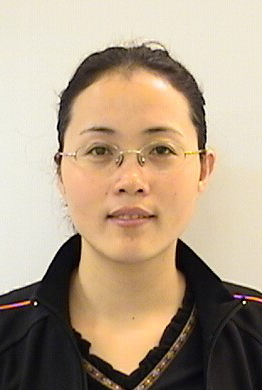}}]{Jun-e Feng} received the Ph.D. degree in cybernetics from Shandong University, Jinan, China, in 2003.

From 2006 to 2007, she was a Visiting Scholar with Massachusetts Institute of Technology, Cambridge, MA, USA, and a Visiting
Scholar with the University of Hong Kong, Hong Kong, in 2013. She is currently a Professor of School of Mathematics with Shandong University. Her research interests include singular systems, multiagent systems, logic control networks, etc
\end{IEEEbiography}

\end{document}